\documentclass[10pt,twoside]{siamltex}
\usepackage{amsfonts,epsfig}

\usepackage{xy}
\input{xy}
\xyoption{all} \xyoption{rotate}

\begin{document}

\bibliographystyle{plain}
\title{
$10$-vertex graphs with cyclic automorphism group of order $4$}

\author{
Peteris\ Daugulis\thanks{Department of Mathematics, Daugavpils
University, Daugavpils, LV-5400, Latvia (peteris.daugulis@du.lv).
} }

\pagestyle{myheadings} \markboth{P.\ Daugulis}{$10$-vertex graphs
with cyclic automorphism group of order $4$} \maketitle

\begin{abstract} We describe computational results about undirected graphs having $10$ vertices and automorphism
group isomorphic to $\mathbb{Z}/4\mathbb{Z}$.

\end{abstract}


\section{Introduction}\

This paper deals with a special case of the problem of finding
undirected graphs $\Gamma$ having a given automorphism group $G$
and minimal number of vertices. All graphs in this paper are
undirected.

Studies of graph automorphism groups may be motivated by the fact
that an isomorphism $G\rightarrow Aut(\Gamma)$ is a representation
of abstract groups as automorphisms of symmetric binary relations
in sets which can be considered as the next complexity level of
discrete objects after sets. Studies of automorphism groups of
graphs started in 1930s with the results of Frucht \cite{F} who
proved constructively in the late 1930s that finite graphs
universally represent finite groups: for any finite group $G$
there is a finite graph $\Gamma=(V,E)$ such that
$Aut(\Gamma)\simeq G$. In the 1970s it was proved by Babai
\cite{B1} constructively that for any finite group $G$ there is a
graph $\Gamma$ such that $Aut(\Gamma)\simeq G$ and $|V(\Gamma)|\le
2|G|$ if $G$ is not cyclic of order $3,4$ or $5$. An estimate
$|V(\Gamma)|\le 3|G|$ and a construction in the three exceptional
cases  was obtained by Sabidussi \cite{S}. Examples of graphs with
$3n$ vertices and cyclic automorphism group
$\mathbb{Z}/n\mathbb{Z}$ are widely known since 1960s, see
\cite{HP}. We can mention that there are $4$ isomorphism types of
graphs with $9$ vertices which form $2$ isomorphism types up to
complementarity. See Babai \cite{B2} for a comprehensive
exposition of this area.

It has been mentioned in the literature that $10$-vertex graphs
with cyclic automorphism group of order $4$ do exist, see
\cite{B1}, \cite{A}. There is an exercise in \cite{H} referring to
such graphs. Our goal is to summarize computational results
related to this problem and popularize results of Meriwether and
Arlinghaus \cite{A}.

We use standard notations of graph theory, see Diestel \cite{D}.
For a graph $\Gamma=(V,E)$ the subgraph induced by $X\subseteq V$
is denoted by $\Gamma[X]$.
\newpage

\section{Main computational results}\

Denote by $F$ the set of isomorphism classes of graphs
$\Gamma=(V,E)$ such that $|V|=10$ and $Aut(\Gamma)\simeq
\mathbb{Z}/4\mathbb{Z}$.

\begin{proposition} Let $\Gamma\in F$.

\begin{enumerate}

\item $|F|$=12. Elements of $F$ form $6$ isomorphism classes up to
complementarity.

\item $18 \le |E(\Gamma)|\le 27$.

\item $3\le \delta(\Gamma) \le 5$,  $4\le \Delta(\Gamma) \le 6$
(minimal and maximal degree).

\item $\Gamma$ has $3$ $Aut(\Gamma)$-orbits with $4$, $4$ and $2$
vertices.

\item $girth(\Gamma)=3$.

\item $3\le \omega(\Gamma) \le 4$ (clique number).

\item $core(\Gamma)$ is isomorphic either to $K_{3}$, $K_{4}$ or
$\Gamma$.

\item $3\le\kappa(\Gamma)\le 5$, $\kappa(\Gamma)=\lambda(\Gamma)$
(vertex and edge connectivity)

\item $2\le diam(\Gamma)\le 3$, $rad(\Gamma)=2$.

\item $3\le \chi(\Gamma)\le 4$ (chromatic number).

\item $F$ contains one planar graph.

\item $F$ contains one Eulerian graph.

\item $\Gamma$ is Hamiltonian.

\item $\Gamma$ is not vertex, edge or distance transitive.


\end{enumerate}

\end{proposition}

\begin{proof} All statement are proved by direct computation.
\end{proof}

\paragraph{Cases}\

We describe two elements of $F$.

\subparagraph{The planar graph}\

 The only planar graph
$\Gamma_{1}\in F$ is shown in Fig.1. It can be thought as embedded
in the $3D$ space, a plane embedding is not given.
$Aut(\Gamma_{1})$ is generated by the vertex permutation
$g=(1,2,3,4)(5,6,7,8)(9,10).$

Subgraphs $\Gamma_{1}[1,2,3,4,5,7,9]$ and
$\Gamma_{1}[1,2,3,4,6,8,10]$ which can be thought as being drawn
above and below the orbit $\Gamma_{1}[1,2,3,4]$ are interchanged
by $g$. The core of $\Gamma_{1}$ is $K_{3}$. The characteristic
polynomial of $\Gamma_{1}$ is
$(x^2-2)^2(x^3-2x^2-8x-4)(x^3+2x^2-4x-4)$.
$$
\xymatrix@R=0.3pc@C=0.3pc{
&&&5\ar@{-}[rrrrd]\ar@{-}[llldddddd]\ar@{-}[rrrdddd]&&&&&&&&&&&\\
&&&&&&&9\ar@{-}[lddd]\ar@{-}[rrrrd]\ar@{-}[rddddddd]&&&&&&&\\
&&&&&&&&&&&7\ar@{-}[rrrdddd]\ar@{-}[llldddddd]&&&\\
&&&&&&&&&&&&&&\\
&&&&&&4\ar@{-}[lllllldd]\ar@{-}[rrrrddddd]\ar@{-}[rrrrrrrrdd]&&&&&&&&\\
&&&&&&&&&&&&&&\\
1\ar@{-}[rrrrrrrrdd]\ar@{-}[rrrrrrrdddd]\ar@{-}[rrrrddddd]&&&&&&&&&&&&&&3\ar@{-}[lllllldd]\ar@{-}[llllddd]\ar@{-}[llllllldddd]\\
&&&&&&&&&&&&&&\\
&&&&&&&&2\ar@{-}[llllddd]&&&&&&\\
&&&&&&&&&&8\ar@{-}[llld]&&&&\\
&&&&&&&10\ar@{-}[llld]&&&&&&&\\
&&&&6&&&&&&&&&&\\
}
$$
\begin{center}

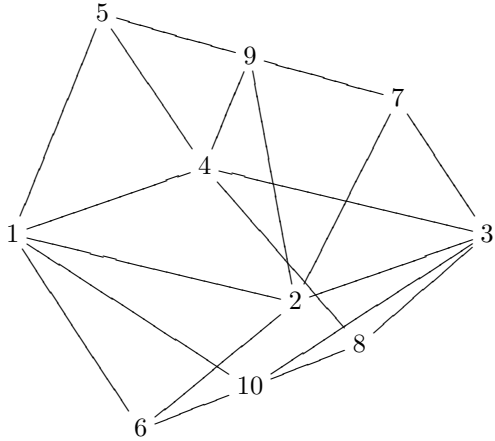
Fig.1.  - $\Gamma_{1}$ the planar graph in $F$.
    \end{center}

\subparagraph{The graph with minimal number of edges}\

 The graph
$\Gamma_{2}\in F$ with minimal number of edges ($18$ edges) is
shown in Fig.2. $Aut(\Gamma_{2})$ is generated by the vertex
permutation $g=(1,2,3,4)(5,6,7,8)(9,10)$. $\Gamma_{2}$ is a core.

$$
\xymatrix@R=0.3pc@C=0.3pc{
&&&5\ar@{-}[rrrrd]\ar@{-}[llldddddd]\ar@{-}[rrrdddd]&&&&&&&&&&&\\
&&&&&&&9\ar@{-}[lllllllddddd]\ar@{-}[rrrrd]\ar@{-}[rrrrrrrddddd]&&&&&&&\\
&&&&&&&&&&&7\ar@{-}[rrrdddd]\ar@{-}[llldddddd]&&&\\
&&&&&&&&&&&&&&\\
&&&&&&4\ar@{-}[rdddddd]\ar@{-}[rrrrddddd]\ar@{-}[rrdddd]&&&&&&&&\\
&&&&&&&&&&&&&&\\
1\ar@{-}[rrrrddddd]\ar@{-}[rrrrrrrrrrrrrr]&&&&&&&&&&&&&&3\ar@{-}[llllddd]\\
&&&&&&&&&&&&&&\\
&&&&&&&&2\ar@{-}[llllddd]\ar@{-}[ldd]&&&&&&\\
&&&&&&&&&&8\ar@{-}[llld]&&&&\\
&&&&&&&10\ar@{-}[llld]&&&&&&&\\
&&&&6&&&&&&&&&&\\
}
$$

\begin{center}

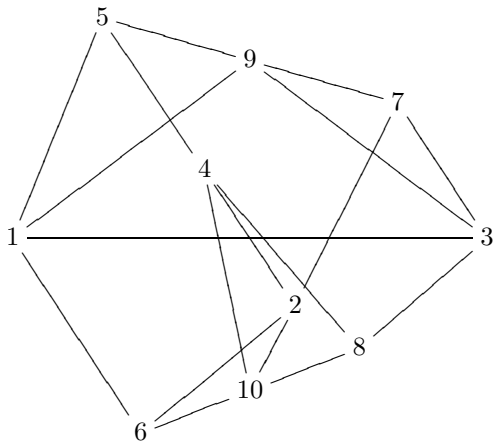
Fig.2.  - $\Gamma_{2}$ - the graph in $F$ with minimal number of
edges.
    \end{center}

\subparagraph{Other graphs}\

 All other graphs in $F$ can be
obtained starting from $\Gamma_{1}$ or $\Gamma_{2}$ and adding or
removing edges in $\Gamma_{2}[1,2,3,4]$, the edge $(9,10)$ and
edges in $\Gamma_{2}[5,6,7,8]$.

\section*{Acknowledgement} Computations were performed using the
computational algebra system MAGMA, see \cite{B3}, graph lists and
the program $nauty$ made public by Brendan McKay, available at
$http://cs.anu.edu.au/~bdm/data/$.


\end{document}